\documentclass[12pt]{amsart}
\usepackage{amsmath,amssymb,amsfonts,amsthm,amsopn,dsfont}
\usepackage{graphics}

\newcommand{\tfa}{time-frequency analysis}

\newcommand{\ft}{Fourier transform}
\newcommand{\stft}{short-time Fourier transform}

\newcommand{\tf}{time-frequency}

\newcommand{\tfs}{time-frequency shift}

\newcommand{\aaf}{A_a^{\varphi_1,\varphi_2}}

\newcommand{\modsp}{modulation space}
\newcommand{\psdo}{pseudodifferential operator}

\newtheorem{theorem}{Theorem}[section]
\newtheorem{corollary}[theorem]{Corollary}
\newtheorem{definition}[theorem]{Definition}

\newtheorem{lemma}[theorem]{Lemma}
\numberwithin{equation}{section}

\newtheorem{proposition}[theorem]{Proposition}
\newtheorem{remark}[theorem]{Remark}

\newcommand{\beqa}{\begin{eqnarray*}}
\newcommand{\eeqa}{\end{eqnarray*}}

\newcommand{\field}[1]{\mathbb{#1}}
\newcommand{\bR}{\field{R}}        
\newcommand{\bZ}{\field{Z}}        
\def\la{\lambda}

 \def\cF{\mathcal{F}}              
 \def\cS{\mathcal{S}}

 \def\cC{\mathcal{C}}

 \def\cN{\mathcal{N}}

\def\vgf{V_gf}

\def\rd{\bR^d}

\def\rdd{{\bR^{2d}}}

\def\lrd{L^2(\rd)}
\def\lrdd{L^2(\rdd)}
\def\zd{\bZ^d}

\def\intrd{\int_{\rd}}
\def\intrdd{\int_{\rdd}}

\def\R{\right)}

\def\<{\left<}
\def\>{\right>}

\def\mv1{M_v^1}

\def\phas{(x,\o )}
\def\mn{(m,n)}
\def\mn'{(m',n')}

\def\o{\omega}

\def\Z{\mathbb{Z}^{d}}

\def\R{\mathbb{R}}
\def\Ren{\mathbb{R}^d}
\def\Renn{\mathbb{R}^{2d}}

\def\sch{\mathcal{S}}

\def\Fur{\mathcal{F}}

\def\f{\varphi}

\def\gaw{A_a^{\f_1,\f_2}}

\def\Sn2{S_{2}(L^{2}(\Ren))}
\def\S1{S_{1}(L^{2}(\Ren))}
\def\sig00{\sigma_{0,0}}

\def\la{\langle}
\def\ra{\rangle}

\begin{document}

\begin{abstract}
We completely characterize
the boundedness on $L^p$
spaces and on Wiener amalgam
spaces of the
\emph{short-time Fourier
transform} (STFT) and of a
special class of
pseudodifferential operators,
called \emph{localization
operators}. Precisely, a
well-known STFT boundedness
result on $L^p$ spaces is
proved to be sharp. Then,
sufficient conditions for the
STFT to be bounded on the
Wiener amalgam spaces
$W(L^p,L^q)$ are given and
their sharpness is shown.
Localization operators are
treated  similarly. Using
different techniques from
those employed in the
literature, we relax the
known sufficient boundedness
conditions for localization
operators on $L^p$ spaces
 and prove the
optimality of our results.
More generally, we prove
sufficient and necessary
conditions
 for such operators to be bounded on Wiener amalgam
 spaces.
\end{abstract}

\title[Sharp Results for the STFT and Localization Operators]{Sharp Continuity Results for the Short-Time Fourier Transform and for Localization Operators}\author{Elena Cordero and Fabio Nicola}
\address{Department of Mathematics,  University of Torino,
Via Carlo Alberto 10, 10123
Torino, Italy}
\address{Dipartimento di Matematica, Politecnico di
Torino, Corso Duca degli
Abruzzi 24, 10129 Torino,
Italy}
\thanks{The second author was partially supported by
the Progetto MIUR
Cofinanziato 2007 ``Analisi
Armonica''}
\email{elena.cordero@unito.it}
\email{fabio.nicola@polito.it}
\date{}

\maketitle

\section{Introduction}

A fundamental issue in
Time-frequency analysis is
the so-called \emph{\stft}
(STFT). Precisely, consider
the linear operators of
translation and modulation
(so-called \tfs s) given by
\begin{equation}
  \label{eqi1}
 T_xf(t)=f(t-x)\quad{\rm and}\quad M_{\o}f(t)= e^{2\pi i \o
  t}f(t) \, .
\end{equation}
 For a non-zero window function $\f$   in
$L^2 (\rd )$,   the \stft\ (STFT) of a signal $f \in
 L^2 (\rd ) $ with respect to the window $\f$ is given by
 \begin{equation}
   \label{eqi2}
   V_\f f(x,\o)=\la f,M_\o T_x \f\ra =\int_{\Ren}
 f(t)\, {\overline {\f(t-x)}} \, e^{-2\pi i\o t}\,dt\, .
 \end{equation}
This  definition can  be extended to  pairs of dual topological
vector spaces and Banach spaces, whose duality, denoted by
$\la\cdot,\cdot\ra$, extends the inner product on $L^2(\rd)$.

The intuitive meaning of the
previous ``\tf"\,
representation can be
summarized as follows. If
$f(t)$ represents a signal
varying in time, its \ft \,
$\hat{f}(\o)$ shows the
distribution of its frequency
$\o$, without any additional
information about ``when"
these frequencies appear. To
overcome this problem, one
may choose a non-negative
window function $\f$
well-localized around the
origin. Then, the information
of the signal $f$ at the
instant $x$ can be obtained
by   shifting the window $\f$
till the  instant $x$ under
consideration, and by
computing the \ft\, of the
product
$f(x)\overline{\f(t-x)}$,
that localizes $f$ around the
instant time $x$. The decay
and smoothness of the STFT
have been widely investigated
 in the framework of $L^p$ spaces \cite{Lieb} and in  certain classes  of Wiener amalgam and modulation
spaces \cite{CG02,CK07,book}.
A basic result, proved in
Lieb \cite{Lieb}, is the
following one:
\vskip0.1truecm

\textit{For every $p\geq2$,
 $f\in L^r(\rd)$, $\f\in L^{r'}(\rd)$,
 with $p'\leq\min\{r,r'\}$, it turns out $V_\f f\in L^p(\rdd)$ and}
$$
\|V_\f f\|_{p} \leq
C\|\f\|_{r'} \|f\|_r.
$$
\noindent (As usual, $p'$ and
$r'$ are the conjugate
exponents of $p$ and $r$,
respectively). The previous
estimate is sharp (see the
subsequent Proposition
\ref{liebs}).\par A new
contribution of the present
paper is the study of the
boundedness of the STFT on
the Wiener amalgam spaces
$W(L^p,L^q)$, $1\leq
p,q\leq\infty$. We recall
that a measurable function
$f$ belongs to $W(L^p,L^q)$
if the following norm
\begin{equation}\label{wienerdis}\|f\|_{W(L^p,L^q)}= \left(\sum_{n\in\zd}\left(\intrd
|f(x)T_n\chi_\mathcal{Q}(x)|^p\right)^{\frac
qp}\right)^{\frac 1q},
\end{equation}
where  $\mathcal{Q}=[0,1)^d$
(with the usual adjustments
if $p=\infty$ or $q=\infty$)
is finite (see \cite{Heil03}
and Section 2 below). In
particular, $W(L^p,L^p)=L^p$.
For heuristic purposes,
functions in $W(L^p,L^q)$ may
be regarded as functions
which are locally in $L^p$
and decay at infinity like a
function in $L^q$.\\ A
simplified version of our
results (see Proposition
\ref{bstftprop}) can be
formulated  as follows.
\vskip0.1truecm \textit{
 Let $\f\in \cS(\R^d)$, $f\in W(L^p,L^q)(\R^d)$,  $1\leq q'\leq p\leq\infty$, $q\geq2$.
Then $V_\f f\in W(L^p,L^q)(\rdd)$, with the uniform
estimate}
\begin{equation}\label{bstft0}
\|V_\f f\|_{W(L^p,L^q)}\leq C_\f \|f\|_{W(L^p,L^q)}.
\end{equation}
\vskip0.1truecm The previous
estimate is \emph{optimal}.
Indeed, if \eqref{bstft0}
holds for a given non-zero
window function $\f \in
\cS(\rd)$, then $ p\geq q'$
and $q\geq2$ (see Proposition
\ref{optimalstft}).
\vskip0.3truecm The other
related topic of the present
paper is the study of
localization operators.
 The name  ``localization operator'' first
appeared in 1988, when
Daubechies \cite{Daube88}
used these operators as a
mathematical tool to localize
a signal on the
time-frequency plane. But
localization operators with
Gaussian windows were already
known in physics: they were
introduced as a quantization
rule by Berezin
\cite{Berezin71} in 1971 (the
so-called Wick operators).
Since their first appearance,
they have been extensively
studied as an important
mathematical tool in signal
analysis and other
applications (see
\cite{CG02,CG04,CR08,RT93,Wong02}
and references therein). We
also recall their employment
as approximation of \psdo s
(wave packets)
\cite{CF78,folland89}.
Localization operators are
also called Toeplitz
operators (see, e.g.,
\cite{DeNo02})  or \stft\,
multipliers \cite{Fei-Now02,
Weisz}. Their definition can
be given by means of the STFT
as follows.
\begin{definition}
The localization operator $\aaf $ with symbol $a\in \cS(\rdd)$ and
windows $\f _1, \f _2\in \cS(\rd)$ is defined   to be
\begin{equation}
  \label{eqi4}
\aaf f(t)=\int_{\Renn}a \phas V_{\f _1}f \phas M_\omega T_x \f _2
(t) \, dx d\o \,, \quad f\in\lrd.
\end{equation}
\end{definition}
The definition extends to more general
classes of symbols $a$, windows
$\f_1,\f_2$, and functions $f$ in
natural way. If $a = \chi _{\Omega }$
for some compact set $\Omega \subseteq
\rdd $ and $\f _1 = \f _2$, then $\aaf
$ is interpreted as the part of $f$
that ``lives on the set $\Omega $'' in
the \tf\ plane. This is why $\aaf $ is
called a   localization operator. If $a
\in \cS '(\rdd )$ and $\f _1, \f _2 \in
\cS (\rd )$, then \eqref{eqi4} is a
well-defined continuous operator from
$\cS (\rd )$ to $\cS ' (\rd )$. If $\f
_1(t)  = \f _2 (t) = e^{-\pi t^2}$,
then $A_a = \aaf $ is the classical
anti-Wick operator and the mapping $a
\to \aaf $ is interpreted as a
quantization
rule~\cite{Berezin71,Shubin91,Wong02}.

 In \cite{CG02,CG04} localization operators are viewed as a
multilinear mapping
\begin{equation}\label{multi}(a,\f _1, \f _2) \mapsto \aaf,
\end{equation} acting on   products of  symbol  and windows
spaces. The dependence of the localization operator $\aaf $ on all
three parameters has been  studied  there in different functional
frameworks.  The results in \cite{CG02} enlarge the ones in the
literature, concerning  $L^p$ spaces \cite{Wong02},
potential and Sobolev spaces \cite{BCG02}, modulation spaces
\cite{Fei-Now02,Toft04,Toftweight}. Other boundedness results for STFT multipliers on $L^p$, modulation, and Wiener amalgam spaces are contained in \cite{Weisz}.

On the footprints of  \cite{CG02}, the study   of localization
operators can be carried to  Gelfand-Shilov spaces and spaces of
ultra-distributions \cite{CRNP}. Finally, the results in
\cite{CR08}  widen \cite{CG02,CG04} interpreting the definition of
$\gaw$ in a weak sense, that is
\begin{equation}\label{anti-Wickg}
 \la \gaw f,g\ra=\la aV_{\f_1}f, V_{\f_2}g\ra=\la
 a,\overline{V_{\f_1}f}\,  V_{\f_2}g\ra,\quad f,g\in\sch(\Ren)\, .
\end{equation}
Here we study the action of
the operators $\gaw$ on the
Lebesgue spaces $L^p$ and on
Wiener amalgam spaces
$W(L^p,L^q)$, when $a$ is
also assumed to belong to
these spaces. For example,
one wonders what is the full
range of exponents $(q,r)$
such that for every $a\in
L^q$ the operator $\gaw$
turns out to be bounded on
$L^r$, for all windows
$\f_1,\f_2$ in reasonable
classes. Partial results in
this connection have been
obtained in
\cite{BOW06,BW04,Wong1999}.
We will give a complete
answer to this question and,
more generally, to a similar
question in the framework of
Wiener amalgam spaces.\par
Our present approach differs
from the
 preceding ones. In particular, the techniques
employed not use  Weyl
operators results as in
\cite{CG02,CRNP}. Indeed, we
rewrite $\gaw$ as an integral
operator. Next, we prove a
Schur-type test for the
boundedness of integral
operators on Wiener amalgam
spaces (see Proposition
\ref{proschur}) and we use it
to obtain sufficient
continuity results for $\gaw$
with symbol $a$ in Wiener
amalgam spaces and acting
either on $L^p$ or on
$W(L^p,L^q)$ (see Theorems
\ref{contlp},
\ref{contwiener} and
\ref{contlp3}). For instance,
the result for Lebesgue
spaces can be simplified as
follows (see Theorem
\ref{contlp3} for more
general windows):
\vskip0.1truecm \textit{Let
$a\in L^q(\rdd)$,
$\f_1,\f_2\in \cS(\Ren)$.
Then, $\gaw$ is bounded  on
$L^r(\Ren)$, for all $1\leq
p\leq\infty$, $\frac1q \geq
|\frac 1r-\frac12|$, with the
uniform estimate}
\begin{equation*}\|\gaw f \|_{r}\leq C_{\f_1,\f_2} \|a\|_{q}\|f\|_r.\end{equation*}
\vskip0.1truecm Moreover, the
boundedness results obtained
on both Lebesgue and Wiener
amalgam spaces reveal to be
sharp: see  Theorems \ref{p3}
and \ref{p4}. In particular,
the optimality for Lebesgue
spaces reads: \vskip0.1truecm
\textit{ Let $\f_1,\f_2\in
\cC^\infty_0(\rd)$,  with
$\f_1(0)=\f_2(0)=1$,
$\f_1\geq0,\ \f_2\geq0$.
Assume that for some $1\leq
q,r\leq\infty$ and every
$a\in L^q(\rdd)$ the operator
$\gaw$ is bounded on $L^r(\rd)$.
Then}
\begin{equation*}
\frac{1}{q}\geq\left|\frac{1}{r}-\frac{1}{2}\right|.
\end{equation*}
We end up  underlining that
similar  methods can be
applied to study the
boundedness of localization
operators on weighted $L^p$
and Wiener amalgam spaces as
well.\par\medskip\noindent
The paper is organized as
follows. Section 2 is devoted
to preliminary definitions
and properties of the
involved function spaces. In
Section 3 we study the
boundedness of the STFT
whereas Section 4 and 5 are
devoted to sufficient and
necessary conditions,
respectively, for boundedness
of localization operators.
Finally, in Section 6 we
present further refinements.

\vskip0.5truecm
\textbf{Notation.} To be definite, let us fix some notation we
shall use later on (and have already used in this Introduction).
We define  $xy=x\cdot y$, the
scalar product on $\Ren$. We define by $\cC^\infty_0(\rd)$ the space of smooth functions on $\rd$ with compact support.
The Schwartz class is denoted by  $\sch(\Ren)$, the space of
tempered distributions by  $\sch'(\Ren)$.   We use the brackets
$\la f,g\ra$ to denote the extension to $\sch (\Ren)\times\sch
'(\Ren)$ of the inner product $\la f,g\ra=\int f(t){\overline
{g(t)}}dt$ on $L^2(\Ren)$.  The Fourier transform is normalized to
be ${\hat
  {f}}(\o)=\Fur f(\o)=\int
f(t)e^{-2\pi i t\o}dt$.
Throughout the paper, we shall use the notation $A\lesssim B$, $A\gtrsim B$ to
indicate $A\leq c B$, $A\geq c B$ respectively,  for a suitable constant $c>0$, whereas $A
\asymp B$ if $A\leq c B$ and $B\leq k A$, for suitable $c,k>0$.
\section{Time-Frequency Methods}

First we summarize some concepts and tools  of \tfa , for an extended
exposition we refer to the  textbooks~\cite{folland89,book}.

\subsection{STFT properties}
The \stft\   (STFT) is defined in \eqref{eqi2}. The STFT  $\vgf $
is defined for $f,g$ in  many possible pairs of Banach spaces or
topological vector spaces.  For instance, it maps $L^2(\rd )
\times L^2(\rd )$ into $\lrdd $ and  $\sch(\Ren)\times\sch(\Ren)$
into $\sch(\Renn)$. Furthermore, it  can be extended  to a map
from $\sch(\Ren)\times\sch'(\Ren)$ into $\sch'(\Renn)$.
\smallskip
The crucial    properties of the STFT  (for proofs, see
\cite{book} and \cite{Lieb}) we shall use in the sequel are the
following.

\begin{lemma}\label{propstft}
Let $f,g,\in L^2(\Ren) $, then we have

(i)   (STFT of time-frequency shifts) For $ y,\xi \in \Ren
  $,  we have
  \begin{equation}
    \label{eql2}
   V_{ {g}}(M_\xi T_y{f})(x,\o)  = e^{-2\pi
     i(\o-\xi)y}(V_gf)(x-y,\o-\xi).
  \end{equation}

(ii)  (Orthogonality
relations for STFT)
\begin{equation}\label{eql4} \| V_{g}f \|_{L^2(\rdd)}=\|f\|_{L^2(\rd)}\|g\|_{L^2(\rd)}.
\end{equation}

(iii)  (Switching $f$ and
$g$)
\begin{equation}\label{switching}
(V_f g)(x,\o)=e^{-2\pi i x\o}
\overline{V_g f(-x,-\o)}.
\end{equation}

(iv) (Fourier transform of a
product of STFTs)
\begin{equation}
(\widehat{V_{\f_1}f
\overline{V_{\f_2}g}})(x,\omega)=(V_{g}f\overline{V_{\f_2}\f_1})(-\omega,x).
\end{equation}

(v)   (STFT of the Fourier transforms)
  \begin{equation}
    \label{eql2op}
   V_{ {g}}f(x,\o)  = e^{-2\pi
     ix\o} V_{\hat g}{\hat f}(\o,-x).
  \end{equation}

\end{lemma}
\begin{proposition}\label{lieb}
Suppose $p\geq 2$, $f\in
L^r(\rd)$, $g\in
L^{r'}(\rd)$, with $p'\leq
\min\{r,r'\}$. Then $V_g f\in
L^p(\rdd)$ and
\begin{equation}\label{liebest}
\|V_gf\|_{p} \leq \sqrt{\frac{p'^{\frac1{p'}}}{p^{\frac1p}}} \,\|g\|_{r'}\|f\|_r.
\end{equation}
\end{proposition}
\noindent The following inequality is proved in \cite[Lemma
11.3.3]{book}:
\begin{lemma}\label{convstft}
Let $g_0,g,\gamma\in\cS(\rd)$ such that $\la \gamma,g\ra\not=0$
and let $f\in\cS'(\rd)$. Then
$$|V_{g_0}f\phas|\leq \frac1{|\la \gamma,g\ra|}(|V_g
f|\ast|V_{g_0}\gamma|)\phas,
$$
for all $\phas \in \rdd$.
\end{lemma}

\subsection{Function Spaces}\label{FS}
For $1\leq p\leq \infty$, recall the $\cF L^p$ spaces, defined by
$$\cF L^p(\rd)=\{f\in\cS'(\rd)\,:\, \exists \,h\in L^p(\rd),\,\hat h=f\};
$$
they are Banach spaces equipped with the norm
\begin{equation}\label{flp}
\| f\|_{\cF L^p}=\|h\|_{L^p},\quad\mbox{with} \,\hat h=f.
\end{equation}
\par

The   mixed-norm space $L^{p,q}(\rdd)$, $1\leq p,q\leq\infty$, consists of all measurable functions on $\rdd$ such that the norm
\begin{equation}\label{normamista}\|F\|_{L^{p,q}}=\left(\intrd\left(\intrd |F\phas|^p dx\right)^{\frac qp} d\o\right)^{\frac1q}
\end{equation}
(with obvious modifications when $p=\infty$ or $q=\infty$)
is finite.
\par

The function spaces  $L^pL^q(\rdd)$, $1\leq p,q\leq\infty$, consists of all measurable functions on $\rdd$ such that the norm
\begin{equation}\label{normainversa}\|F\|_{L^pL^q}=\left(\intrd\left(\intrd |F\phas|^q d\o\right)^{\frac pq} dx\right)^{\frac1q}
\end{equation}
(with obvious modifications when $p=\infty$ or $q=\infty$)
is finite. Notice that, for $p=q$, we have  $L^pL^p(\rdd)=L^{p,p}(\rdd)=L^p(\rdd)$.

\textbf{Wiener amalgam spaces}. We
briefly recall the definition and the
main properties of Wiener amalgam
spaces. We refer to
\cite{feichtinger80,feichtinger83,
feichtinger90,Fei98,fournier-stewart85,Heil03}
for details.\par
 Let $g \in
\cC_0^\infty$ be a test function that
satisfies $\|g\|_{L^2}=1$. We will
refer to $g$ as a window function. Let
$B$ one of the following Banach spaces:
$L^p, \cF L^p$, $L^{p,q}$, $L^pL^q$,
$1\leq p,q\leq \infty$. Let $C$ be one
of the following Banach spaces: $L^p$,
$L^{p,q}$, $L^pL^q$, $1\leq
p,q\leq\infty$. For any given temperate
distribution $f$ which is locally in
$B$ (i.e. $g f\in B$, $\forall
g\in\cC_0^\infty$), we set $f_B(x)=\|
fT_x g\|_B$.

The {\it Wiener amalgam space} $W(B,C)$
with local component $B$ and global
component  $C$ is defined as the space
of all temperate distributions $f$
locally in $B$ such that $f_B\in C$.
Endowed with the norm
$\|f\|_{W(B,C)}=\|f_B\|_C$, $W(B,C)$ is
a Banach space. Moreover, different
choices of $g\in \cC_0^\infty$ generate
the same space and yield equivalent
norms.

If  $B=\Fur L^1$ (the Fourier
algebra),  the space of
admissible windows for the
Wiener amalgam spaces $W(\Fur
L^1,C)$ can be enlarged to
the so-called Feichtinger
algebra $W(\Fur L^1,L^1)$.
Recall  that the Schwartz
class $\sch$
  is dense in $W(\Fur L^1,L^1)$.\par
   An equivalent norm for the space  $W(L^p,L^q)$ is provided by \eqref{wienerdis}.\par
We use the following
definition of mixed Wiener
amalgam norms. Given a
measurable function $F$ on $\rdd$ we
set
\[
\|F\|_{W(L^{p_1},L^{q_1})W(L^{p_2},L^{q_2})}= \|
\|F(x,\cdot)\|_{W(L^{p_2},L^{q_2})}\|_{W(L^{p_1},L^{q_1})}.
\]
The following properties of
Wiener amalgam spaces  will
be frequently used in the
sequel.
\begin{lemma}\label{WA}
  Let $B_i$, $C_i$, $i=1,2,3$, be Banach spaces  such that $W(B_i,C_i)$ are well
  defined. Then,
  \begin{itemize}
  \item[(i)] \emph{Convolution.}
  If $B_1\ast B_2\hookrightarrow B_3$ and $C_1\ast
  C_2\hookrightarrow C_3$, we have
  \begin{equation}\label{conv0}
  W(B_1,C_1)\ast W(B_2,C_2)\hookrightarrow W(B_3,C_3).
  \end{equation}
  \item[(ii)]\emph{Inclusions.} If $B_1 \hookrightarrow B_2$ and $C_1 \hookrightarrow C_2$,
   \begin{equation*}
   W(B_1,C_1)\hookrightarrow W(B_2,C_2).
  \end{equation*}
  \noindent Moreover, the inclusion of $B_1$ into $B_2$ need only hold ``locally'' and the inclusion of $C_1 $ into $C_2$  ``globally''.
   In particular, for $1\leq p_i,q_i\leq\infty$, $i=1,2$, we have
  \begin{equation}\label{lp}
  p_1\geq p_2\,\mbox{and}\,\, q_1\leq q_2\,\Longrightarrow W(L^{p_1},L^{q_1})\hookrightarrow
  W(L^{p_2},L^{q_2}).
  \end{equation}
  \item[(iii)]\emph{Complex interpolation.} For $0<\theta<1$, we
  have
\[
  [W(B_1,C_1),W(B_2,C_2)]_{[\theta]}=W\left([B_1,B_2]_{[\theta]},[C_1,C_2]_{[\theta]}\right),
  \]
if $C_1$ or $C_2$ has
absolutely continuous norm.
    \item[(iv)] \emph{Duality.}
    If $B',C'$ are the topological dual spaces of the Banach spaces $B,C$ respectively, and
    the space of test functions $\cC_0^\infty$ is dense in both $B$ and $C$, then
\begin{equation}\label{duality}
W(B,C)'=W(B',C').
\end{equation}
 \item[(v)] \emph{Hausdorff-Young.} If $1\leq p,q\leq 2$ then
  \begin{equation}\label{HY}
\cF(W(L^p,L^q))\hookrightarrow W(L^{q'},L^{p'})
\end{equation}
(local and global properties are interchanged on the Fourier side).
 \item[(vi)] \emph{Pointwise products.}
  If $B_1\cdot B_2\hookrightarrow B_3$ and $C_1\cdot
  C_2\hookrightarrow C_3$, we have
  \begin{equation}\label{point0}
  W(B_1,C_1)\cdot W(B_2,C_2)\hookrightarrow W(B_3,C_3).
  \end{equation}
  \end{itemize}
  \end{lemma}
The following result will be useful in the sequel.
\begin{lemma}[{\cite[Lemma 5.4]{CN3}}]
Let $\f_{a+ib}(x)=e^{-\pi
(a+ib)|x|^2}$, $a>0, b\in\R$, $x\in\R^d$. Then,
\begin{equation}
\label{dd1}
\|\f_{a+ib}\|_{W(L^{p},L^{q})}
\asymp a^{-
\frac{d}{2q}}(a+1)^{\frac{d}{2}\left(\frac{1}{q}-\frac{1}{p}\right)}.
\end{equation}
\end{lemma}
Finally we establish a useful estimate
for Wiener amalgam spaces, based on the
 the
classical Bernstein's
inequality (see, e.g.,
\cite{wolff}), that we are
going to recall first. Let
$B(x_0,R)$ be the ball of
center $x_0\in\rd$ and radius
$R>0$ in $\R^d$.
\begin{lemma}[Bernstein's inequality]\label{bernstein} Let $f\in\cS'(\rd)$
such that  $\hat{f}$ is supported in
$B(x_0,R)$, and let  $1\leq p\leq
q\leq\infty$. Then, there exists a
positive constant $C$, independent of
$f$, $x_0$, $R$, $p,$ $q$, such that
\begin{equation}\label{bernsteineq}
\|f\|_q\leq C
R^{d\left(\frac1p-\frac1q\right)}\|f\|_p.
\end{equation}
\end{lemma}
\begin{proposition}\label{new}
Let $1\leq p\leq q\leq\infty$. For
every $R>0$, there exists a constant
$C_R>0$ such that, for every
$f\in\mathcal{S}'(\R^d)$ whose Fourier
transform is supported in any ball of
radius $R$, it turns out
\[
\|f\|_{W(L^q,L^p)}\leq C_R\|f\|_{p}.
\]
\end{proposition}
\begin{proof}
Choose a Schwartz function $g$ whose
Fourier transform $\hat{g}$ has compact
support in $B(0,1)$, as window function
arising in the definition of the norm
in $W(L^q,L^p)$. Then, the function
$(T_x g) f$, $x\in\R^d$, has Fourier
transform supported in a ball of radius
$R+1$. Therefore it follows from
Bernstein's inequality (Lemma
\ref{bernstein}) that
\[
\|(T_x g) f\|_{q}\leq C_R\| (T_x g)
f\|_{p}.
\]
Taking the $L^p$-norm with respect to
$x$ gives the conclusion.
\end{proof}

\vskip0.3truecm

\textbf{Modulation spaces}. For their basic properties we refer to
\cite{F1,book}. \par
Given a non-zero window $g\in\sch(\Ren)$ and $1\leq p,q\leq
\infty$, the {\it
  modulation space} $M^{p,q}(\Ren)$ consists of all tempered
distributions $f\in\sch'(\Ren)$ such that the STFT, defined in \eqref{eqi2}, fulfills $V_gf\in L^{p,q}(\Renn )$. The norm on $M^{p,q}$ is
$$
\|f\|_{M^{p,q}}=\|V_gf\|_{L^{p,q}}=\left(\int_{\Ren}
  \left(\int_{\Ren}|V_gf(x,\o)|^p\,
    dx\right)^{q/p}d\o\right)^{1/p}  \, .
$$
If $p=q$, we write $M^p$ instead of $M^{p,p}$.

$M^{p,q}$ is a Banach space
whose definition is independent of the choice of the window $g$. Moreover,
if $g \in   M^1 \setminus \{0\}$, then
$\|V_gf \|_{L^{p,q}}$ is an
equivalent norm for $M^{p,q}(\Ren)$.

 Among the properties of \modsp s, we record that
$M^{2,2}=L^2$, $M^{p_1,q_1}\hookrightarrow M^{p_2,q_2}$, if
$p_1\leq p_2$ and $q_1\leq q_2$. If $1\leq p,q<\infty$, then
$(M^{p,q})'=M^{p',q'}$.\par

Modulation spaces and Wiener amalgam spaces are
closely related: for $p=q$, we have
\begin{equation}\label{Wienermod}\|f\|_{W({\cF
L}^p,L^p)}=\left(\int_{\rd}\int_{\rd}|V_{g}f\phas|^p\, m\phas^p
dx\,d\o\right)^{1/p}\asymp  \|f\|_{M^p}.
\end{equation}

Finally, the next results  will be
useful in the sequel.
\begin{proposition}[{\cite[Proposition 3.4]{CK07}}]\label{mpcha}
Let $g\in M^1(\R^d)$ and
$1\leq p \leq \infty$ be
given. Then \,\,$f\in
M^p(\rd)$ if and only if $
V_g f\in M^p(\rdd)$\,\, with
\begin{equation}\label{equivMp} \|V_g f\|_{M^p}\asymp \|g\|_{M^p}\|f\|_{M^p}.\end{equation}
\end{proposition}
\begin{lemma}[{\cite[Lemma 4.1]{CG02}}]\label{amalg}
Let  $1\leq p,q\leq\infty$. If  $f\in  M^{p,q}(\Ren)$ and   $g\in
M^{1}(\Ren)$, then  $V_gf \in  W(\Fur
L^1,L^{p,q})(\Renn)$ with norm estimate
\begin{equation}\label{amalgeq}
\| V_gf\|_{W(\Fur L^1,L^{p,q})}\lesssim \|f\|_{
  M^{p,q}}\|g\|_{M^{1}}\, .
\end{equation}
\end{lemma}
We finally recall a
characterization of
modulation spaces. Let $\psi
\in \cS(\R^d)$ be such that
\begin{equation}\label{2.1}
\mathrm{supp}\, \psi \subset (-1,1)^d
\qquad \text{and} \qquad \sum_{k \in
\zd}\psi(\o-k)=1 \quad \text{for all
$\o \in \R^d$}.
\end{equation}
\begin{proposition}{\rm (\cite{triebel83})}\label{caratterizzazione}
Let $1 \le p,q \le \infty$.
An equivalent norm in
$M^{p,q}$ is given by
\[
\|f\|_{M^{p,q}}
\asymp\left(\sum_{k \in \zd}
\|\psi(D-k)f\|_{L^p}^q
\right)^{1/q},
\]
where
$\psi(D-k)f:=\Fur^{-1}(\hat{f}
T_k\psi)$.
\end{proposition}

\section{Boundedness of the STFT}
To prove  sharp results for the STFT, we need the
following lemmata.
\begin{lemma}\label{l1} Let
$h\in\cC^\infty_0(\R^d)$, and
consider the family of
functions
\[
h_\lambda(x)=h(x) e^{-\pi
i\lambda|x|^2}, \qquad
\lambda\geq1.
\]
Then, if $q\geq2$,
\begin{equation}\label{ub}
\|\widehat{h_\lambda}\|_{q}\lesssim
\lambda^{\frac{d}{q}-\frac{d}{2}},
\end{equation}
whereas, if $1\leq q<2$,
\begin{equation}\label{lb}
\|\widehat{h_\lambda}\|_{q}\gtrsim
\lambda^{\frac{d}{q}-\frac{d}{2}}.
\end{equation}
\end{lemma}
\begin{proof}
The upper bound \eqref{ub} is
shown, e.g., in \cite[page
28]{wolff}. The lower bound
\eqref{lb} follows from
\eqref{ub} since, for $q<2$,
\[
0\not=C=\|h_\lambda\|^2_{2}=\|\widehat{h_\lambda}\|^2_{2}\leq
\|\widehat{h_\lambda}\|_{{q'}}
\|\widehat{h_\lambda}\|_{q}\lesssim\lambda^{\frac{d}{q'}-\frac{d}{2}}
\|\widehat{h_\lambda}\|_{q}.
\]
\end{proof}

\begin{lemma}\label{ll1}
Let $\f(t)=e^{-\pi|t|^2}$ and consider the family of
functions $\f_\lambda(t)=e^{-\pi\lambda|t|^2}$, $\lambda>0$. We
have
\[
V_\f \f_\lambda(x,\omega)=(\lambda+1)^{-\frac{d}{2}}
e^{-\frac{\pi(\lambda |x|^2+|\omega|^2)}{\lambda+1}}e^{-\frac{2\pi
i}{\lambda+1}x\omega}.
\]
\end{lemma}
\begin{proof}
We have
\begin{align*}
V_\f \f_\lambda(x,\omega)&=\int e^{-\pi\lambda|t|^2}
e^{-\pi|t-x|^2} e^{-2\pi i
\omega t}\,dt=\int e^{-\pi(\lambda+1)\left|t-\frac{x}{\lambda+1}\right|^2}e^{-\frac{\pi\lambda|x|^2}{\lambda+1}}e^{-2\pi
i \omega t}\,dt\\
&=e^{-\frac{\pi\lambda|x|^2}{\lambda+1}}\Fur\left(
T_{\frac{x}{\lambda+1}} e^{-\pi(\lambda+1)|\cdot|^2}\right)(\omega)=(\lambda+1)^{-\frac{d}{2}}e^{-\frac{\pi\lambda|x|^2}{\lambda+1}}M_{-\frac{x}{\lambda+1}} e^{-\frac{\pi|\omega|^2}{\lambda+1}}\\
&=(\lambda+1)^{-\frac{d}{2}} e^{-\frac{\pi(\lambda
|x|^2+|\omega|^2)}{\lambda+1}}e^{-\frac{2\pi
i}{\lambda+1}x\omega},
\end{align*}
as desired.
\end{proof}

Then, the result of Proposition \ref{lieb} is sharp, as
explained below.
\begin{proposition}\label{liebs}
Suppose that, for some $1\leq p,r\leq\infty$, the following
inequality holds
\begin{equation}\label{liebests}
\|V_gf\|_{p} \leq C\|g\|_{r'}
\,\|f\|_r,\quad \forall
f,g\in \cS(\R^d).
\end{equation}  Then  $p'\leq\min\{r,r'\}$
(in particular, $p\geq 2$).
\end{proposition}
\begin{proof}
 We start by proving  the constraint $p'\leq r$.
 Let $\phi(t)=e^{-\pi|t|^2}$ and $\f_\lambda(t)=
e^{-\pi\lambda |t|^2}$. Then,
by Lemma \ref{ll1}, we have
\[
|V_\f
\f_\lambda(x,\omega)|=(\lambda+1)^{-\frac{d}{2}}
e^{-\frac{\pi(\lambda
|x|^2+|\omega|^2)}{\lambda+1}}.
\]
Hence,
\begin{equation}\label{f1l}
\|V_\f
\f_\lambda\|_{p}\asymp
(\lambda+1)^{-\frac{d}{2}}\left(\frac{\lambda}{\lambda+1}\right)^{-\frac{d}{2p}}\left(\frac{1}{\lambda+1}\right)^{-\frac{d}{2p}}.
\end{equation}
Since $\|\f_\lambda\|_r\asymp \lambda^{-d/(2r)}$, writing
\eqref{liebests} for $f=\f_\lambda$, $g=\phi$ and letting
$\lambda\to+\infty$, we get $p'\leq r$, as desired.\par The
constraint $p'\leq r'$ follows similarly by switching the role of
$f$ and $g$, i.e., by using \eqref{switching}.\par
 Of
course, $p'\leq\min\{r,r'\}$
implies $p\geq 2$. Here is an
alternative direct proof.
Assume, by contradiction,
that $p\leq2$, and test the
estimate \eqref{liebests} on
the functions $f=h_\lambda$,
$\lambda\geq 1$, of Lemma
\ref{l1}. It is well-known
(see, e.g.,
\cite{fe89-1,feichtinger90,grobner,kasso07})
that,  for distributions
supported in a fixed compact
subset, the norm in $M^p$ is
equivalent to the norm in
$\Fur L^p$. Hence, using
\eqref{lb},
\[
\|V_\f h_\lambda\|_p\asymp \|h_\lambda\|_{M^p}\asymp \|{\widehat{h_\lambda}}\|_{p}\gtrsim \lambda^{\frac dp-\frac d2}.\]
Since $\|h_\lambda\|_r=\|h\|_r=C_h$, letting $\lambda\to+\infty$ in \eqref{liebests}, the claim $p\geq 2$ follows.
\end{proof}
\vskip0.3truecm We now focus on the
boundedness of the STFT on the Wiener
amalgam spaces $W(L^p,L^q)$. We will
show that, if the windows are in $M^1$,
the optimal range of admissible pairs
$(1/q,1/p)$ is the shadowed region of
Figure 1.
  \vspace{1.2cm}
  \small
 \begin{center}
           \includegraphics{figWickbis.1}
            \\
           $ $
\end{center}
\normalsize
           \begin{center}{Figure 1:
            Estimate \eqref{bstft} holds for all pairs $(1/q,1/p)$ in the shadowed region.}
           \end{center}
\subsection{Sufficient boundedness
conditions for the STFT on $W(L^p,L^q)$.}
\begin{proposition}\label{bstftprop}
Let $\f\in M^1(\R^d)$, $f\in W(L^p,L^q)(\rd)$,
$1\leq q'\leq p\leq\infty$,
$q\geq2$. Then $V_\f f \in W(L^p,L^q)(\R^{2d})$, with the
uniform estimate
\begin{equation}\label{bstft}
\|V_\f
f\|_{W(L^p,L^q)}\lesssim
\|f\|_{W(L^p,L^q)}\|\f\|_{M^1}.
\end{equation}
\end{proposition}
\begin{proof}
By interpolation it suffices
to prove the desired result
when $p=q'$, $2\leq
q\leq\infty$, and when
$p=\infty$, $2\leq
q\leq\infty$.\par Let $p=q'$,
$2\leq q\leq\infty$. We have
\[
\|V_\f
f\|_{W(L^{q'},L^q)}\lesssim
\|V_\f f\|_{q},
\]
because of the embedding
$L^q\hookrightarrow
W(L^{q'},L^q)$ (recall,
$q\geq2$). By Lemma \ref{convstft}, Young's Inequality and the equality $M^q= W(\Fur
L^q,L^q)$, we obtain
\[
\|V_\f f\|_{q} \lesssim\|\f\|_{M^1}\|f\|_{W(\Fur
L^q,L^q)}\lesssim
\|\f\|_{M^1}\|f\|_{W(L^{q'},L^q)},\]
where the last inequality
is due to the embedding
$W(L^{q'},L^q)\hookrightarrow
W(\Fur L^q,L^q)$ (that is a
consequence of the
Hausdorff-Young
Inequality).\par Let now
$p=\infty$, $2\leq
q\leq\infty$. We have
\[
\|V_\f
f\|_{W(L^\infty,L^q)}\lesssim
\|V_\f f\|_{W(\Fur
L^1,L^q)}\lesssim \|f\|_{M^q}
\|\f\|_{M^1},
\]
where the second inequality
is due to \eqref{amalgeq}. The last
expression is
\[
\|f\|_{W(\Fur
L^q,L^q)}\|\f\|_{M^1}\lesssim
\|f\|_{W(L^{q'},L^q)}\|\f\|_{M^1}\lesssim
\|f\|_{W(L^\infty,L^q)}\|\f\|_{M^1}.
\]
This concludes the proof.
\end{proof}

\subsection{Necessary boundedness
conditions for the STFT on $W(L^p,L^q)$.}
In what follows, we shall prove the sharpness of the results obtained above.
\begin{proposition}\label{optimalstft}
Let $g\in M^1\setminus\{0\}$.
Suppose that, for some $1\leq
p,q\leq\infty$, $C>0$, the
estimate
\begin{equation}\label{m11}
\|V_g f\|_{W(L^p,L^q)}\leq C
\|f\|_{W(L^p,L^q)},\qquad
\forall f\in\cS(\R^d),
\end{equation}
holds. Then $p\geq q'$ and
$q\geq2$.
\end{proposition}
\begin{proof}
We claim that we may just
consider the case of the
window $\f(t)=e^{-\pi|t|^2}$.
Indeed, if $g$ is a non-zero
window function in $M^1$,
Lemma \ref{convstft} and
\eqref{conv0} give
\[
\|V_{g}f\|_{W(L^p,L^q)}\leq
\frac1{\|\f\|^2_2}\|V_\f
f\|_{W(L^p,L^q)}\|V_g\f\|_{L^1}
\]
and
\[
\|V_{\f}f\|_{W(L^p,L^q)}\leq
\frac1{\|g\|_2^2}\|V_g
f\|_{W(L^p,L^q)}\|V_\f g
\|_{L^1},
\]
so that
$\|V_{g}f\|_{W(L^p,L^q)}\asymp
\|V_{\f}f\|_{W(L^p,L^q)}$.
 Hence, from now on, we assume
 $g=\f$.\par
 First we prove that
$p\geq q'$. Let
$\f_\lambda(t)=
e^{-\pi\lambda|t|^2}$. Then,
by Lemma \ref{ll1}, we have
$|V_\f
\f_\lambda(x,\omega)|=(\lambda+1)^{-\frac{d}{2}}
e^{-\frac{\pi(\lambda
|x|^2+|\omega|^2)}{\lambda+1}}.$
By taking $g\phas=(g_1\otimes
g_2)\phas$ as window function
in the definition of the
Wiener amalgam norm, one sees
that
\begin{equation}\label{f1}
\|V_\f
\f_\lambda\|_{W(L^p,L^q)}\asymp
(\lambda+1)^{-\frac{d}{2}}
\|e^{-\frac{\pi\lambda|\cdot|^2}
{\lambda+1}}\|_{W(L^p,L^q)}\cdot\|e^{-\frac{\pi|\cdot|^2}
{\lambda+1}}\|_{W(L^p,L^q)}.
\end{equation}
By the estimate \eqref{dd1}, it
turns out, for
$\lambda\geq1$,
\begin{equation}\label{fi1}
\|\f_\lambda\|_{W(L^p,L^q)}\asymp
\lambda^{-\frac{d}{2q}}
(\lambda+1)^
{\frac{d}{2}\left(\frac{1}{q}-\frac{1}{p}\right)}\asymp\lambda^{-\frac{d}{2p}},
\end{equation}
and
\[
\|e^{-\frac{\pi\lambda|\cdot|^2}
{\lambda+1}}\|_{W(L^p,L^q)}\asymp
\left(\frac{\lambda}{\lambda+1}\right)^{-\frac{d}{2q}}
\left(\frac{\lambda}{\lambda+1}+1\right)^
{\frac{d}{2}\left(\frac{1}{q}-\frac{1}{p}\right)},
\]
\[
\|e^{-\frac{\pi|\cdot|^2}
{\lambda+1}}\|_{W(L^p,L^q)}\asymp
\left(\frac{1}{\lambda+1}\right)^{-\frac{d}{2q}}
\left(\frac{1}{\lambda+1}+1\right)^
{\frac{d}{2}\left(\frac{1}{q}-\frac{1}{p}\right)},
\]
which, by \eqref{f1}, give
\[
\|V_\f
\f_\lambda\|_{W(L^p,L^q)}\asymp
\lambda^{-\frac{d}{2q'}}.
\]
Letting $\lambda\to+\infty$, the previous estimate and \eqref{fi1}  yield
 $p\geq q'$.\par Let us now
prove the condition $q\geq2$.
If $p\leq q$ this follows
from the condition $p\geq
q'$, so that we can suppose
$p> q$. Then
$W(L^p,L^q)\hookrightarrow
L^q$, and from \eqref{m11} we
deduce
\[
\|f\|_{M^q}\asymp\|V_\f
f\|_{L^q}\lesssim
\|f\|_{W(L^p,L^q)}.
\]
We now argue as in the proof of Proposition \ref{liebs} and  test this estimate on
the functions $f=h_\lambda$
in Lemma \ref{l1}. This  implies
\[
\|V_\f h_\lambda\|_q\asymp\|h_\lambda\|_{M^q}\asymp\|\widehat{h_\lambda}\|_{L^q}\lesssim
\|h_\lambda\|_{W(L^p,L^q)}=C_{p,q}.
\]
Since the constant $C_{p,q}$ is independent
of $\lambda$, it follows from
 \eqref{lb},   letting
 $\lambda\to+\infty$, that
 $q\geq2$.
\end{proof}
\section{Sufficient Boundedness Conditions for Localization Operators}
In what follows, we study the boundedness of localization operators on $L^p$  and on Wiener amalgam spaces.

First of all, we prove
sufficient boundedness
results for localization
operators acting on Lebesgue
spaces. This is in fact a
particular case of Theorem
\ref{contwiener}, but we
establish it separately for
the benefit of the reader.
\begin{theorem}\label{contlp}
Let  $a\in W(L^p,L^q)(\rdd)$, $\f_1,\f_2\in M^{1}(\Ren)$.
Then
$\gaw$ is bounded  on  $L^r(\Ren)$,
 for all $1\leq p,q,r\leq\infty$,
$\frac1q \geq  |\frac 1r-\frac12|$,
with the uniform estimate
\begin{equation}\label{bdlp}\|\gaw\|_{B(L^r)}\lesssim  \|a\|_{W(L^p,L^q)}\|\f_1\|_
{M^{1}}\|\f_2\|_{ M^{1}}.\end{equation}
\end{theorem}
Figure 2 illustrates the range of exponents $(1/r,1/q)$ for the
boundedness of $\gaw$.
  \vspace{1.2cm}

 \begin{center}
           \includegraphics{figWick.1}
            \\
           $ $
\end{center}
           \begin{center}{Figure 2:
            Estimate \eqref{bdlp} holds for all pairs $(1/r,1/q)$ in the shadowed region.}
           \end{center}
\begin{proof}
By the inclusion relations for Wiener amalgam spaces, it suffices
to show the claim for $p=1$. For $q=\infty$, $r=2$, the result was
already proved in \cite[Theorem 1.1]{CG02}. Indeed, using the
inclusion $L^1\subset \cF L^\infty$ and the inclusion relations
for Wiener amalgam spaces, we have $W(L^1,L^\infty)\subset W(\cF
L^\infty,L^\infty)= M^\infty$. Hence, the symbol $a$ is in the
modulation space $M^\infty$ and, consequently, the operator $\gaw$
is bounded on $M^2=L^2$.

Let us prove the thesis in the cases
$1\leq q\leq2$ and $1\leq r\leq\infty$,
so that, using the interpolation
diagram of Figure 1, we attain the
desired result.

By the  Schur's test (see, e.g.,
\cite[Lemma 6.2.1]{book}), it is enough
to show that the kernel of $\gaw$
belongs to the spaces $L^{1,\infty}$
and $L^\infty L^1$, defined in
\eqref{normamista} and
\eqref{normainversa}, respectively. Let
us start with the first case. $\gaw$ is
the integral operator with kernel
\begin{align}
K(x,y)&=\intrdd a(t,\o) M_\omega T_t \f _2(x)M_{-\omega}  \overline{T_t\f _1(y)}dt d\omega\label{kernelaw0}\\
&=  \intrd \cF_2 a(t,y-x) \overline{T_t\f _1(y)}T_t\f_2(x) dt\\
&=\intrd \cF(\overline{\f _1(y-\cdot)}\f_2(x-\cdot))(\o)\cF a(\o,y-x) d\o\\
&=\intrd V_{\f_2^*}
(\overline{T_y\f _1^*})\phas
\cF a(\o,y-x)
d\o,\label{kernelaw}
\end{align}
with the notation
$f^*(t)=f(-t)$. Using the
pointwise multiplication
properties for Wiener amalgam
spaces \eqref{point0}, and
setting
$\psi_y(x,\o)=(\o,y-x)$ we
obtain
\[
\intrd |K(x,y)| dx\leq \|
V_{\f_2^*} (\overline{T_y\f
_1^*})\|_{W(L^q,L^1)}\|(\Fur
a)\circ\psi_y\|_{W(L^{q'},L^\infty)}
\]
On the other hand, it is
easily seen that $\|(\Fur
a)\circ\psi_y\|_{W(L^{q'},L^\infty)}\asymp
\|(\Fur
a)\|_{W(L^{q'},L^\infty)}$,
uniformly with respect to
$y$. Hence, by the inclusion
$W(L^1, L^q)\subset \cF
W(L^{q'},L^\infty)$ (see the
Hausdorff-Young property
\eqref{HY}), for $1\leq q\leq
2$, we deduce
\begin{align*}
\intrd |K(x,y)| dx &\leq \|
V_{\f_2^*} (\overline{T_y\f
_1^*})\|_{W(L^q,L^1)}\|a\|_{
W(L^1,L^q)} \lesssim \|
V_{\f_2^*} (\overline{T_y\f
_1^*})\|_{M^1}\|a\|_{
W(L^1,L^q)}\\ &\lesssim
\|\overline{T_y\f
_1^*}\|_{M^1}\|\f_2^*\|_{M^1}\|a\|_{
W(L^1,L^q)}=\|\f
_1\|_{M^1}\|\f_2\|_{M^1}\|a\|_{
W(L^1,L^q)},
\end{align*}
where we have used $M^1=W(\cF
L^1,L^1)\subset W(L^q,L^1)$,
by the local inclusion $\cF
L^1\subset L^q$ (see Lemma
\ref{WA}, item (ii)), and
Proposition \ref{mpcha} with
$p=1$. Hence,
\[
\|K\|_{L^{1,\infty}}\lesssim
\|a\|_{ W(L^1,L^q)} \|\f
_1\|_{M^1}\|\f_2\|_{M^1}.
\]
Similarly, one obtains
$$\|K\|_{L^\infty L^{1}}\lesssim \|a\|_{W(L^1,L^q)} \|\f
_1\|_{M^1}\|\f_2\|_{M^1}$$
and the estimate \eqref{bdlp}
follows.
\end{proof}

To study the boundedness of  localization operators  on Wiener
amalgam spaces, we  rewrite them as integral operators and use the
following Schur-type test for integral operators on Wiener amalgam
spaces.

\begin{proposition}\label{proschur}
Consider an integral operator defined  by
$$D_Kf(x)=\intrd K(x,y) f(y) \,dy.
$$
If the kernel $K$ belongs to the following spaces (see definitions \eqref{normamista} and \eqref{normainversa}):
\begin{equation}\label{cond1}
K\in W(L^{1,\infty}, L^{\infty}L^1)(\rdd)\cap W(L^\infty L^1,
L^{1,\infty})(\rdd)
\end{equation}
and
\begin{equation}\label{cond2}K\in
L^\infty L^1(\rdd)\cap L^{1,\infty}(\rdd),
\end{equation} then
the operator $D_K$ is continuous on $W(L^p,L^q)(\rd)$, for every $1\leq
p,q\leq\infty$.
\end{proposition}

\begin{proof}
Let $\mathcal{Q}$ denote the unit cube $[0,1)^d$. Then,  we can
write any  function $f$ on $\rd$ as $f(x)=\sum_{n\in\zd}
f(x)T_n\chi_\mathcal{Q}(x)$. Moreover, it is straightforward to
show that $D_K f(x)=\sum_{n\in\zd} D_K(fT_n\chi_\mathcal{Q})(x)$.

We first study the boundedness of $D_K$ on $W(L^1,L^\infty)$.
Using the expression above, we have
$$\|D_K f\|_{W(L^1,L^\infty)}=\sup_{k\in\zd}\|D_K(\sum_{n\in\zd}fT_n\chi_\mathcal{Q})T_k\chi_\mathcal{Q}\|_{1}\leq
\sup_{k\in\zd}\sum_{n\in\zd}\|D_K(fT_n\chi_\mathcal{Q})T_k\chi_\mathcal{Q}\|_1.
$$
The equality
$T_n\chi_\mathcal{Q}(y)=T_n\chi^2_\mathcal{Q}(y)=(T_n\chi_\mathcal{Q}(y))^2$,
yields
\begin{align*}
\|D_K(fT_n\chi_\mathcal{Q})T_k\chi_\mathcal{Q}\|_1&\leq\intrd\left|\intrd K(x,y)f(y)T_n\chi_\mathcal{Q}(y) T_k\chi_\mathcal{Q}(x)dxdy\right|\\
&\leq\|T_n\chi_\mathcal{Q}f\|_1 \sup_{y\in\rd}\left(\intrd
T_n\chi_\mathcal{Q}(y)|K(x,y)|T_k\chi_\mathcal{Q}(x)\,dx\right).
\end{align*}
Hence,
\begin{align*}
\sup_{k\in\zd}\sum_{n\in\zd}\|D_K(fT_n\chi_\mathcal{Q})T_k\chi_\mathcal{Q}\|_1&\leq
\sup_{n\in\zd}\|fT_n\chi_\mathcal{Q}\|_1\sup_{k\in\zd}\sum_{n\in\zd}\|K(T_k\chi_\mathcal{Q}\otimes T_n\chi_\mathcal{Q})\|_{L^{1,\infty}}\\
&=\|f\|_{W(L^1,L^\infty)}\|K\|_{W(L^{1,\infty},L^{\infty}L^1)}.
\end{align*}

Using similar arguments, one easily  obtains
$$\|D_Kf\|_{W(L^\infty,L^1)}\leq \|f\|_{W(L^\infty,L^1)}\|K\|_{W(L^\infty L^1,L^{1,\infty})}.
$$
 Since the statement
holds for $p=q$ by the classical Schur's test, the operator $D_K$
is bounded on $W(L^p,L^p)=L^p$.  By  complex interpolation between
$(p,p)$ and  $(p,q)=(1,\infty)$, we get the boundedness of $D_K$
on every $W(L^p,L^q)$, $1\leq p< q<\infty$. The complex
interpolation between $(p,p)$ and $(p,q)=(\infty,1)$ yields the
boundedness of $D_K$ on every $W(L^p,L^q)$, $1\leq q<
p\leq\infty$. It remains to study the cases $q=\infty$ and
$1<p<\infty$.
 For
these cases we argue by duality: it suffices to verify that
\begin{equation}\label{ddua}|\la D_K f,g\ra|\leq \|f\|_{W(L^p,L^\infty)}\|g\|_{W(L^{p'},L^1)}, \quad\forall f\in W(L^p,L^\infty),\,\forall g\in W(L^{p'},L^1).
\end{equation}
Now
\begin{align*}
|\la D_K f,g\ra|&\leq \intrd \intrd |K(x,y)f(y)g(x)|dx dy =\intrd\left(\intrd |K(x,y) g(x)| dx\right)|f(y)| dy\\
&\leq
\|D_{\tilde{K}}g\|_{W(L^{p'},L^1)}\|f\|_{W(L^p,L^\infty)},\end{align*}
where $D_{\tilde{K}}$ is the operator with kernel
$\tilde{K}(x,y)=|K(y,x)|$. Since it satisfies the same assumptions
as $D_K$, it is continuous on $W(L^{p'},L^1)$. This yields
\eqref{ddua}.
\end{proof}

Observe that the condition \eqref{cond2} is the  assumption in the classical Schur's test
for the continuity of $D_K$ on $L^p$ spaces.  Furthermore, the
condition $K\in W(L^{1,\infty}, L^{\infty,1}) \cap W(L^{\infty,1},
L^{1,\infty})$ implies  the assumption \eqref{cond1}.

\begin{proposition}\label{bo}
Let $a\in
W(L^1,L^\infty)(\rdd)$,
$\f_1,\f_2\in M^1(\R^d)$.
Then the operator $\gaw$ is
bounded on $W(L^2,L^s)(\R^d)$
for every $1\leq
s\leq\infty$, with the
uniform estimate
\[
\|\gaw\|_{B(W(L^2,L^s))}\lesssim
\|a\|_{W(L^1,L^\infty)}\|\f_1\|_{M^1}
\|\f_2\|_{M^1}.
\]
\end{proposition}

\begin{proof}
We have to prove the estimate
\[
|\langle \gaw f,g\rangle|\leq
C\|a\|_{W(L^1,L^\infty)}\|f\|_{W(L^2,L^s)}\|g\|_{W(L^2,L^{s'})}\|\f_1\|_{M^1}
\|\f_2\|_{M^1},
\]
for every $f\in W(L^2,L^s)$,
$g\in W(L^2,L^{s'})$. Using the weak definition \eqref{anti-Wickg}, we can write
\[
\langle \gaw
f,g\rangle=\int_{\R^{2d}}
a(x,\omega)V_{\f_1}f(x,\omega)\overline{V_{\f_2}g(x,\omega)}\,dx\,d\omega.
\]
 By Lemma \ref{WA} (ii),
\[
|\langle \gaw f,g\rangle|\leq
\|a\|_{W(L^1,L^\infty)}\|V_{\f_1}f\overline{V_{\f_2}g}\|_{W(L^1,L^\infty)}.
\]
Write $f=\sum_{k\in\zd}
fT_{k}\chi_\mathcal{Q}$ and
$g=\sum_{h\in\zd}
gT_{h}\chi_\mathcal{Q}$.
Moreover, choose
$\psi\in\cS(\R^d)$ satisfying
\eqref{2.1} and write
$\f_1=\sum_{l\in\zd}\f_{1,l}$,
$\f_2=\sum_{m\in\zd}\f_{2,m}$,
with $\f_{1,l}=\psi(D-l)\f_1$
and $\f_{2,m}=\psi(D-m)\f_2$
(with the notation in
Proposition
\ref{caratterizzazione}). We
deduce
\begin{equation}\label{13uno}
|\langle \gaw f,g\rangle|\leq
\|a\|_{W(L^1,L^\infty)}\sum_{m,l\in\zd}\sum_{k,h\in\zd}\|V_{\f_{1,l}}(f
T_k\chi_\mathcal{Q})
\overline{V_{\f_{2,m}}(g
T_h\chi_\mathcal{Q})}\|_{W(L^1,L^\infty)}.
\end{equation}
 By applying Lemma
 \ref{propstft}
(iv) and (v) we see that
\begin{align*}
\big|\Fur\left(V_{\f_{1,l}}(f
T_k\chi_\mathcal{Q})
\overline{V_{\f_{2,m}}(g
T_h\chi_\mathcal{Q})}\right)&(x,\omega)\big|\\
&=|V_{g T_h\chi_\mathcal{Q}}(f
T_k\chi_\mathcal{Q})(-\omega,x)
\overline{V_{\f_{2,m}}(\f_{1,l})}(-\omega,x)|\\
&=|V_{g
T_h\chi_\mathcal{Q}}(f
T_k\chi_\mathcal{Q})(-\omega,x)
\overline{V_{\widehat{\f_2}T_{m}\psi}(\widehat{\f_1}T_{l}\psi)}(x,\omega)|.
\end{align*}
The key observation is now
the following one: if
$\gamma_1,\gamma_2\in
L^2(\R^d)$ are supported in
balls of radius, say, $R$,
then the STFT
$V_{\gamma_1}\gamma_2(x,\omega)$
has support in a strip of the
type $B(y_0,2R)\times\R^d$,
for some $y_0\in\R^d$. This
follows immediately from the
definition of STFT. Hence the
computation above shows that
the expression
\[V_{\f_{1,l}}(f T_k\chi_\mathcal{Q})
\overline{V_{\f_{2,m}}(g
T_h\chi_\mathcal{Q})}
\]
has Fourier transform
supported in a ball in
$\R^{2d}$ whose radius is
independent of $k,h,m,l$.
Hence it follows from
\eqref{13uno} and Proposition
\ref{new}
 that
\begin{equation}\label{13due}
|\langle \gaw f,g\rangle|\lesssim
\|a\|_{W(L^1,L^\infty)}\sum_{m,l\in\zd}\sum_{k,h\in\zd}\|V_{\f_{1,l}}(f
T_k\chi_\mathcal{Q})
\overline{V_{\f_{2,m}}(g
T_h\chi_\mathcal{Q})}\|_{1}.
\end{equation}
As a consequence of
Cauchy-Schwarz's inequality
and Parseval's formula, this
last expression is
\begin{align*}
&\leq
\|a\|_{W(L^1,L^\infty)}\sum_{m,l\in\zd}\sum_{k,h\in\zd}\int_{\R^d}\|V_{\f_{1,l}}(fT_k\chi_\mathcal{Q})(x,\cdot)\|_2\|V_{\f_{2,m}}(gT_h\chi_\mathcal{Q})(x,\cdot)\|_2\,dx\\
&=\|a\|_{W(L^1,L^\infty)}\sum_{m,l\in\zd}\sum_{k,h\in\zd}
\int_{\R^d}\|fT_k\chi_\mathcal{Q}T_x\f_{1,l}\|_{2}\|gT_h\chi_\mathcal{Q}T_x\f_{2,m}\|_2\,dx\\
&\leq
\|a\|_{W(L^1,L^\infty)}\sum_{k,h\in\zd}
\|f T_k\chi_\mathcal{Q}\|_{2}\|gT_h
\chi_\mathcal{Q}\|_2\sum_{m,l\in\zd}\int_{\R^d}
\|T_k\chi_\mathcal{Q}T_x\f_{1,l}\|_\infty
\|T_h
\chi_\mathcal{Q}T_x\f_{2,m}\|_\infty\,dx.
\end{align*}
We will prove that
\begin{equation*}
\int_{\R^d}
\|T_k\chi_\mathcal{Q}T_x\f_{1,l}\|_\infty
\|T_h
\chi_\mathcal{Q}T_x\f_{2,m}\|_\infty\,dx=\int_{\R^d}
\|\chi_\mathcal{Q}T_x\f_{1,l}\|_\infty
\|\chi_\mathcal{Q}T_{x+k-h}\f_{2,m}\|_\infty\,dx=v_{k-h,l,m},
\end{equation*}
for a sequence $v=v_{k,l,m}\in
l^1(\mathbb{Z}^{3d})$, satisfying
\begin{equation}\label{21b}
\|v\|_{l^1(\mathbb{Z}^{3d})}\lesssim\|\f_1\|_{M^1}\|\f_2\|_{M^1}.
\end{equation}
 Assuming \eqref{21b},  by H\"older's
 and Young's inequality,
\begin{align*}
|\langle &\gaw f,g\rangle|\lesssim
\|a\|_{W(L^1,L^\infty)}\sum_{m,l\in\zd}\sum_{k,h\in\zd}
v_{k-h,l,m}\|gT_h
\chi_\mathcal{Q}\|_2\|f
T_k\chi_\mathcal{Q}\|_{2}\\
&\lesssim \|a\|_{W(L^1,L^\infty)}
\sum_{m,l\in\zd}
\left(\sum_{k\in\zd}\left(\sum_{h\in
\zd} v_{k-h,l,m}\|gT_h
\chi_\mathcal{Q}\|_2\right)^{s'}\right)^{1/s'}\left(\sum_{k\in
\zd}\|f
T_k\chi_\mathcal{Q}\|_{2}^s\right)^{1/s}\\
&\leq \|a\|_{W(L^1,L^\infty)}
\sum_{m,l\in\zd}
\left(\sum_{k\in\zd}v_{k,l,m}\right)
\left(\sum_{h\in \zd}\|g
T_h\chi_\mathcal{Q}\|_{2}^{s'}\right)^{1/s'}
\left(\sum_{k\in \zd}\|f
T_k\chi_\mathcal{Q}\|_{2}^s\right)^{1/s}
\end{align*}
and the desired estimate follows by
\eqref{21b}.\par Let us now prove
\eqref{21b}. Indeed,
\begin{align}\label{13c}
\sum_{k\in\Z} v_{k,m,l}&=\int_{\R^d}
\|\chi_\mathcal{Q}T_x\f_{1,l}\|_\infty \sum_{k\in\Z} \|\chi_\mathcal{Q}T_{x+k-h}\f_{2,m}\|_\infty\,dx\\
&\asymp\|\f_{2,m}\|_{W(L^\infty,L^1)}\int_{\R^d}
\|\chi_\mathcal{Q}T_x\f_{1,l}\|_\infty
\,dx\asymp\|\f_{1,l}\|_{W(L^\infty,L^1)}
\|\f_{2,m}\|_{W(L^\infty,L^1)}\\
&\lesssim \|\f_{1,l}\|_{1}
\|\f_{2,m}\|_1,
\end{align}
where the last estimate
follows from Proposition
\ref{new}. Then, by
Proposition
\ref{caratterizzazione}, we
have
\[
\sum_{k,l,m\in\zd}
v_{k,m,n}\lesssim
\sum_{l\in\zd}\|\f_{1,l}\|_{1}\sum_{m\in\zd}\|\f_{1,m}\|_{1}\asymp\|\f_1\|_{M^1}\|\f_2\|_{M^1},
\]
which concludes the proof.
\end{proof}
\begin{remark}\rm
An argument similar to that
in the proof of Proposition
\ref{bo} (but simpler) shows
that, if $a\in
L^\infty(\rdd)$ and
$\f_1,\f_2\in
W(L^\infty,L^1)(\R^d)$, then
$\gaw$ is bounded on
$W(L^2,L^s)(\R^d)$ for every
$1\leq s\leq\infty$.
\end{remark}

Propositions \ref{proschur}
and \ref{bo}  are the main
ingredients to study the
boundedness of localization
operators on Wiener amalgam
spaces.
\begin{theorem}\label{contwiener}
Let   $a\in
W(L^p,L^q)(\rdd)$,
$\f_1,\f_2\in M^{1}(\R^d)$.
Then $\gaw$ is bounded  on
$W(L^{r},L^{s})(\Ren)$, for
all
\begin{equation}\label{indi}
1\leq p,q,r,s\leq\infty,\quad
\frac1q \geq \left|\frac
1{r}-\frac12\right|,
\end{equation}
with the uniform estimate
\begin{equation}\label{bdwin}\|\gaw\|_{B(W(L^{r},L^{s}))}\lesssim  \|a\|_{W(L^p,L^q)}\|\f_1\|_
{M^{1}}\|\f_2\|_{
M^{1}}.\end{equation}
\end{theorem}
\begin{proof} We use similar arguments to those of Theorem \ref{contlp}. Again, it is enough to prove the claim for $p=1$. We shall show that
 the cases $1\leq q\leq 2$ yield the continuity of $\gaw $ on $W(L^{r},L^{s})$ for every $1\leq r,s\leq\infty$, so that, by complex
 interpolation with the case $(q,r)=(\infty,2)$, considered in Proposition \ref{bo}
 (see Lemma \ref{WA} (iii) and Figure 1) we obtain the desired boundedness of $\gaw$ under the conditions \eqref{indi}, if $s<\infty$.
 The remaining cases, when
 $s=\infty$ and $q>2$ (and therefore
 $r>1$)
 follows by duality, for
 then
 $W(L^r,L^s)=W(L^{r'},L^{s'})'$
 (Proposition \ref{WA}
 (iv)),  and
$(\gaw)^*=A^{\f_2,\f_1}_{\bar
a}$.\par Hence, let $1\leq
q\leq 2$. In order to prove
the boundedness  of $\gaw $
on $W(L^{r},L^{s})$ for every
$1\leq r,s\leq\infty$ we use
the Schur's test for Wiener
amalgam spaces given by
Proposition
\ref{proschur}.\par We
already verified
\eqref{cond2} for the
integral kernel $K$ of
$\gaw$, which was computed in
\eqref{kernelaw}. Let us
verify that $K\in
W(L^{1,\infty},L^{\infty}L^1)$.
To this end, we estimate the
kernel as follows:
\begin{align*}
 |K(x,y)| &\leq\intrd |V_{\f_2^*} (\overline{T_y\f _1^*})\phas \cF a(\o,y-x) |d\o\\
 &=\intrd  |V_{\f_2^*} \overline{\f _1^*}(x-y,\o) \cF a(\o,y-x)
 |d\o,
\end{align*}
since the STFT fulfills $V_g
(T_y f)\phas=e^{-2\pi i\o
y}V_gf(x-y,\o) $, see
\eqref{eql2}.  For sake of
simplicity, we set
$\Phi:=V_{\f_2^*}
\overline{\f _1^*}$.
Moreover, we introduce the
coordinate transformation
$\tau F(t,u)=F(-u,t)$. So
that the kernel $K$ can be
controlled from above by
$$|K(x,y)|\leq \intrd |\tau\Phi(\o,y-x)|  \,|\cF a(\o,y-x) |d\o.$$
Let us estimate
$\|K\|_{W(L^{1,\infty},L^\infty
L^1)}$. Settin $B(t):=\intrd
|\tau\Phi(\o,t)|  \,|\cF
a(\o,t) |d\o$, we obtain
\begin{align*}
\|K\|_{W(L^{1,\infty},L^\infty
L^1)}&=\sup_{k\in\zd}\sum_{n\in\zd}\sup_{y\in\rd}\intrd
T_k\chi_{\mathcal{Q}}(x)
T_n\chi_{\mathcal{Q}}(y)|K(x,y)|
dx\\
&\leq\sup_{k\in\zd}\sum_{n\in\zd}\sup_{y\in\rd}\intrd
T_k\chi_{\mathcal{Q}}(x)T_n\chi_{\mathcal{Q}}(y)|B(y-x)|\,dx\\
&\leq \sup_{k\in\zd}\sum_{n\in\zd}\sup_{y\in\rd}T_n\chi_{\mathcal{Q}}(y)(T_k\chi_{\mathcal{Q}}\ast|B|)(y)\\
&=\sup_{k\in\zd}\|T_k\chi_{\mathcal{Q}}\ast|B|\|_{W(L^\infty,L^1)}.
\end{align*}
Using $L^\infty\ast
L^1\hookrightarrow L^\infty$,
$L^1\ast L^1 \hookrightarrow
L^1$ and the convolution
relations for Wiener amalgam
spaces in Lemma \ref{WA},
Item (i), we obtain
$$\|T_k\chi_{\mathcal{Q}}\ast|B|\|_{W(L^\infty,L^1)}\leq \|T_k\chi_{\mathcal{Q}}\|_{W(L^\infty,L^1)}\|B\|_{W(L^1,L^1)}\leq \|B\|_1,
$$
since $\|
T_k\chi_{\mathcal{Q}}\|_{W(L^\infty,L^1)}=1$
and $W(L^1,L^1)=L^1$. We are
left to estimate $\|B\|_1$.
Precisely,
\begin{align*}
\|B\|_1&=\intrd\intrd |\tau\Phi(\o,t)|  \,|\cF a(\o,t) |d\o dt\\
&\leq
\|\tau\Phi\|_{W(L^q,L^1)}\|\cF a\|_{W(L^{q'},L^\infty)}\\
&\leq\|\tau\Phi\|_{M^1}\|
a\|_{W(L^{1},L^q)}\\
&=\|\Phi\|_{M^1}\| a\|_{W(L^{1},L^q)}\\
&\lesssim
\|\f_1\|_{M^1}\|\f_2\|_{M^1}\|
a\|_{W(L^{1},L^q)},
\end{align*}
where we  used the inclusion
$\cF L^1\subset L^q$, which
holds locally and yields
$W(\cF L^1,L^1)\subset
W(L^q,L^1)$ (Lemma \ref{WA},
item (ii)) and the equality
$W(\cF L^1,L^1)=M^1$ and,
finally, \eqref{equivMp}.

Similar  arguments yield
$$\|K\|_{W(L^\infty L^{1},L^{1,\infty})}\lesssim \|\f_1\|_{M^1}\|\f_2\|_{M^1}\| a\|_{W(L^{1},L^q)}, $$
and the desired result
follows.
\end{proof}

\section{Necessary Boundedness Conditions for Localization Operators}
We need the following version
of Lemma \ref{l1} for Wiener
amalgam spaces.
\begin{lemma}\label{l2} With the notation of
Lemma \ref{l1}, we have, for
$q\geq 2$,
\[
\|\widehat{h_\lambda}\|_{W(L^p,L^q)}\lesssim
\lambda^{\frac{d}{q}-\frac{d}{2}},\qquad
\lambda\geq1.
\]
\end{lemma}
\begin{proof}
 When $p\leq q$, the result follows at once from Lemma
\ref{l1}, using the embedding
$L^q\hookrightarrow
W(L^p,L^q)$.\par When $p>q$
it suffices to apply
Proposition \ref{new} and
Lemma \ref{l1} again.
\end{proof}
\begin{proposition}\label{p1}
Let $\f_1,\f_2\in
\cC^\infty_0(\rd)$, $\chi\in
\cC^\infty(\rd)$, with
$\f_1(0)=\f_2(0)=\chi(0)=1$,
$\f_1\geq0,\ \f_2\geq0,\
\chi\geq0$. If the estimate
\begin{equation}\label{a00}
\|\chi \gaw f\|_{r}\leq C
\|a\|_{W(L^p,L^q)}\|f\|_{W(L^{s_1},L^{s_2})},\,\quad
\forall f\in \cS(\R^d),\
\forall a\in\cS(\R^{2d}),
\end{equation}
holds for some $1\leq
p,q,r,s_1,s_2\leq\infty$,
then
\begin{equation}\label{a01}
\frac{1}{q}\geq\frac{1}{2}-\frac{1}{r}.
\end{equation}
\end{proposition}
\begin{proof}
We can assume $q\geq2$,
otherwise the conclusion is
trivially true.
 Consider a
real-valued function
$h\in\cC^\infty_0(\R^d)$,
with $h(0)=1$, $h\geq0$. Let
\[
h_\lambda(x)=h(x) e^{-\pi
i\lambda|x|^2}, \qquad
\lambda\geq1.
\]
We test the estimate
\eqref{a00} on
$f=\overline{h_\lambda}$ and
\[
a(x,\o)=a_\lambda(x,\o)=h(x)
(\Fur^{-1}h_\lambda)(\o).
\]
Clearly,
$\|h_\lambda\|_{W(L^{s_1},L^{s_2})}$
is independent of $\lambda$.
On the other hand, by Lemma
\ref{l2} we have
\[
\|a_\lambda\|_{W(L^p,L^q)}=
\|h\|_{W(L^p,L^q)}\|\Fur^{-1}h_\lambda\|_{W(L^p,L^q)}
\lesssim
\lambda^{\frac{d}{q}-\frac{d}{2}}.
\]
Hence, if we prove that
\begin{equation}\label{b2}
\|\chi
A_{a_\lambda}^{\f_1,\f_2}
f\|_{r}\gtrsim
\lambda^{-\frac{d}{r}},
\end{equation}
then \eqref{a01} follows from
\eqref{a00}, by letting
$\lambda\to+\infty$.\par Let
us verify \eqref{b2}. It
suffices to prove that
\[
|\chi(x)
(A_{a_\lambda}^{\f_1,\f_2}
\overline{h_\lambda})(x)|\gtrsim1,\qquad
{\rm for}\ |x|\leq
\lambda^{-1},
\]
and for all
$\lambda\geq\lambda_0$ large
enough. To this end, observe
that, by \eqref{kernelaw0},
$A_{{a_\lambda}}^{\f_1,\f_2}$
is an integral operator with
kernel
\[
K(x,y)=\int h(t)h(y-x)e^{-\pi
i\lambda|y-x|^2}
{\f_1(y-t)}\f_2(x-t)\,dt.
\]
Hence
\[
|\chi(x)
(A_{a_\lambda}^{\f_1,\f_2}
\overline{h_\lambda})(x)|=\left|\chi(x)\iint
e^{2\pi i \lambda xy}
h(t)h(y-x) {\f_1(y-t)}
\f_2(x-t)h(y)\,dt\,dy\right|.
\]
Since $\f_1,\f_2$ have
compact support, $K(x,y)$ is
identically zero if
$|x-y|\geq C$, for a
convenient $C>0$. Moreover,
we are considering the last
integral for $|x|\leq
\lambda^{-1}\leq1$ only, so
that we can in fact integrate
over $|y|\leq C+1$ only. On
the other hand, we have
\[
{\rm Re}\, e^{2\pi i \lambda
xy}=\cos(2\pi\lambda
xy)\geq\frac{1}{2},\ {\rm
if}\ |x|\leq \lambda^{-1},\
|y|\leq C+1,
\]
for all
$\lambda\geq\lambda_0$ large
enough. It follows that
\begin{multline}
\left|\chi(x)\iint e^{2\pi i
\lambda xy} h(t)h(y-x)
{\f_1(y-t)}
\f_2(x-t)h(y)\,dt\,dy\right|\\
\geq
\frac{1}{2}\chi(x)\left|\iint
h(t)h(y-x) {\f_1(y-t)}
\f_2(x-t)h(y)\,dt\,dy\right|.
\end{multline}
Now, this last expression is
clearly bounded from below by
a constant $\delta>0$, for
$x$ in a neighbourhood of the
origin. This follows by
continuity, since for $x=0$
the integral is strictly
positive by our assumptions
on $\f_1,\f_2,h$, and
$\chi(0)=1$.\par This
concludes the proof.
\end{proof}

\begin{theorem}\label{p3}
Let $\f_1,\f_2\in
\cC^\infty_0(\rd)$, with
$\f_1(0)=\f_2(0)=1$,
$\f_1\geq0,\ \f_2\geq0$. If,
for some $1\leq
p,q,r,s\leq\infty$, the
estimate
\begin{equation}\label{a0}
\|\gaw f\|_{W(L^r,L^s)}\leq C
\|a\|_{W(L^p,L^q)}\|f\|_{W(L^r,L^s)},\,\quad
\forall f\in \cS(\R^d),\
\forall a\in\cS(\R^{2d}),
\end{equation}
holds, then
\begin{equation}\label{a1}
\frac{1}{q}\geq\left|\frac{1}{r}-\frac{1}{2}\right|.
\end{equation}
\end{theorem}
\begin{proof}
 We can suppose $q\geq2$ (otherwise
\eqref{a1} is trivially
satisfied). We assume $r\geq
2$, the case $1\leq r<2$
follows by duality, for
$(\gaw)^*=A^{\f_2,\f_1}_{\bar
a}$.\par
  Let $\chi\in
\cC^\infty_0(\R^d)$,
$\chi\geq 0$, $\chi(0)=1$.
Then, \eqref{a0} implies that
\[
\|\chi \gaw
f\|_{W(L^r,L^s)}\leq C
\|a\|_{W(L^p,L^q)}\|f\|_{W(L^r,L^s)}\,\quad
\forall f\in \cS(\R^d),\
\forall a\in\cS(\R^{2d}).
\]
For functions $u$ supported
in a fixed compact subset we
have
$\|u\|_{W(L^r,L^s)}\asymp
\|u\|_{r}$, so that we have
\[
\|\chi\gaw f\|_{r}\leq C
\|a\|_{W(L^p,L^q)}\|f\|_{W(L^r,L^s)},\,\quad
\forall f\in \cS(\R^d),\
\forall a\in\cS(\R^{2d}).
\]
Then \eqref{a1} follows from
Proposition \ref{p1}.
\end{proof}
\begin{theorem}\label{p4}
Let $\f_1,\f_2\in
\cC^\infty_0(\rd)$, with
$\f_1(0)=\f_2(0)=1$,
$\f_1\geq0,\ \f_2\geq0$.
Assume that for some $1\leq
p,q,r,s\leq\infty$ and every
$a\in W(L^p,L^q)(\rdd)$ the
operator $\gaw$ is bounded on
$W(L^r,L^s)(\R^d)$. Then
\eqref{a1} must hold.
\end{theorem}
\begin{proof}
Consider first the case
$1\leq r,s<\infty$. By
Theorem \ref{p3} and the
Closed Graph Theorem it
suffices to prove that the
map
\[
a\in W(L^p,L^q)\longmapsto
\gaw\in B(W(L^r,L^s))
\]
has closed graph. To this
end, consider a sequence
$a_n\to a$ in $W(L^p,L^q)$
and such that
$A_{a_n}^{\f_1,\f_2}\to A$ in
$B(W(L^r,L^s))$. We verify
that $\gaw=A$, that is
$\langle\gaw
f,g\rangle=\langle
Af,g\rangle$ for every
$f,g\in\cS(\R^d)$, because
$\cS(\R^d)$ is dense in
$W(L^r,L^s)$. Clearly,
$\langle A_{a_n}^{\f_1,\f_2}
f,g\rangle\to\langle
Af,g\rangle $. On the other
hand, since $V_{\f_1} f$,
$V_{\f_2} g$ are Schwartz
functions,
\[
\langle A_{a_n}^{\f_1,\f_2}
f,g\rangle=\int a_n(x,\o)
V_{\f_1}
f(x,\o)\overline{V_{\f_2}g}(x,\o)\,dx\,d\o\]
tends to
\[
\int a(x,\o) V_{\f_1}
f(x,\o)\overline{V_{\f_2}g}(x,\o)\,dx\,d\o=\langle
\gaw f,g\rangle.
\]
It remains to consider the
cases $r=\infty$ or
$s=\infty$. Here one can
argue by contradiction.
Namely, suppose that $\gaw$
is bounded on $W(L^r,L^s)$
for every $a\in W(L^p,L^q)$,
for some $q,r$ that do not
satisfy \eqref{a1}. Hence,
$\displaystyle\frac1q<\left|\frac1r-\frac12\right|$.
Suppose $r=\infty$, which
implies $q>2$. By
interpolating with the case
$(r,s)=(2,2)$ (being fixed
the pair $(p,q)$), one would
obtain the boundedness of
$\gaw$ for every $a\in
W(L^p,L^q)$, on certain
$W(L^{\tilde{r}},L^{\tilde{s}})$,
$\tilde{r},\tilde{s}<\infty$,
with $\tilde{r}$ arbitrarily
large, so that \eqref{a1}
(with $\tilde{r}$ in place of
$r$) fails. Similarly, for
$s=\infty$ and interpolating
with the case $(r,s)=(2,2)$,
one would obtain
  the
boundedness of $\gaw$ for
every $a\in W(L^p,L^q)$, on
certain
$W(L^{\tilde{r}},L^{\tilde{s}})$,
$\tilde{r},\tilde{s}<\infty$,
with $\frac1q<
\left|\frac{1}{\tilde{r}}-\frac12\right|$,
so that \eqref{a1} fails
again. This contradicts what
we showed in the first part
of the present proof.
\end{proof}

\section{Further results}
In this section we discuss a
slight improvement of Theorem
\ref{contlp}, in the special
case $p=q$, i.e., for symbols
$a\in L^q$. Precisely, we
consider windows in Lebesgue
spaces only, rather than in
$M^1$. To state the outcome
of our study,  we need the
following definition. A
measurable function $F$ on
$\rdd$ belongs to the space
$W(L^{p_1},L^{q_1})\Fur
W(L^{p_2},L^{q_2})(\rdd)$ if the
mapping $x\longmapsto
\|F(x,\cdot)\|_{\Fur
W(L^{p_2},L^{q_2})}$ is in
the Wiener amalgam space
$W(L^{p_1},L^{q_1})(\rd)$, with
norm
$$\| F\|_{W(L^{p_1},L^{q_1})\Fur W(L^{p_2},L^{q_2})}=\|\,\|F(x,\cdot)\|_{\Fur
W(L^{p_2},L^{q_2})}\|_{W(L^{p_1},L^{q_1})}.
$$
Then our first result reads
as follows.
\begin{proposition}\label{contlp2}
Let $1\leq p_1,p_2,q_1,q_2,
r\leq\infty$. Let
\[
a\in W(L^{p_1},L^{q_1}) \Fur
W(L^{p_2},L^{q_2})(\rdd),
\]
and
\[
\f_1,\f_2\in
 \cN(\rd):=W(L^{p'_1},L^{q'_1})(\rd)\cap
W(L^{p'_2},L^{q'_2})(\rd).
\] Then
$\gaw$ is bounded on
$L^r(\Ren)$, with the uniform
estimate
\begin{equation}\label{bdlp2}\|\gaw\|_{B(L^r)}
\lesssim
\|a\|_{W(L^{p_1},L^{q_1})_x
\Fur
W(L^{p_2},L^{q_2})_\omega}\|\f_1\|_
{\cN}\|\f_2\|_{\cN}.\end{equation}
\end{proposition}
\begin{proof}
The proof is similar to that
of Theorem \ref{contlp}, so
that we only present the main
points. We apply Schur's
test. Hence we verify that
the integral kernel \[
K(x,y)=\int\cF_2 a(t,y-x)
\overline{T_t\f
_1(y)}T_t\f_2(x)\,dt
\]
of $\gaw$ is in $L^\infty
L^1\cap L^{1,\infty}$.\par
Let us verify that $K\in
L^\infty L^1$, the other
condition is similar. By the
kernel expression above  and
H\"older's Inequality we have
\begin{align*}
\|K\|_{L^\infty
L^1}&\leq\sup_{x\in\R^d}\int
|T_t\f_2(x)|\int |T_t\f_1(y)|
|\cF_2 a(t,y-x)|\,dy\,dt\\
&\leq
\|\f_1\|_{W(L^{p'_2},L^{q'_2})}\sup_{x\in\R^d}\int
|T_t\f_2(x)|
\|a(t,\cdot)\|_{\Fur
W(L^{p_2},L^{q_2})}\,dt\\
&\leq
\|\f_1\|_{W(L^{p'_2},L^{q'_2})}\|\f_2\|_{W(L^{p'_1},L^{q'_1})}\|a\|_{W(L^{p_1},L^{q_1})_x\Fur
W(L^{p_2},L^{q_2})},
\end{align*}
which is the desired
estimate.
\end{proof}

Finally, we can prove the
boundedness result for
symbols in $L^q$.\par
\begin{theorem}\label{contlp3}
Let $\cN_q=L^q(\R^d)\cap
L^{q'}(\R^d)$ if $1\leq q\leq
2$ and $\cN_q=L^2(\R^d)$ if
$2<q\leq\infty$. Let $a\in
L^q(\rdd)$, $\f_1,\f_2\in
\cN_q$. Then $\gaw$ is
bounded on $L^r(\Ren)$, for
all $\frac1q \geq |\frac
1r-\frac12|$, with the
uniform estimate
\begin{equation}\label{bdlp3}\|\gaw\|_{B(L^r)}
\lesssim \|a\|_{q}\|\f_1\|_
{\cN_q}\|\f_2\|_{\cN_q}.\end{equation}
\end{theorem}
\begin{proof}
By interpolation with the
well-known case $q=\infty$,
$r=2$ \cite{Wong1999}, it
suffices to prove the desired
result for $q\leq2$. But this
follows from Proposition
\ref{contlp2} with
$p_1=q_1=q$, $p_2=q_2=q'$
(recall,
$W(L^{q},L^{q})=L^q$,
$W(L^{q'},L^{q'})=L^{q'}$),
combined with the
Hausdorff-Young inequality
$L^q\hookrightarrow\Fur
L^{q'}$.
\end{proof}

Observe that Theorem
\ref{contlp3} extends the
result contained in Theorem
3.3 of \cite{BOW06} to larger
index sets and this range is
sharp. Indeed, for  $s=r$ and
$p=q$, Theorems \ref{p3} and
\ref{p4} read as follows.
\begin{corollary}\label{p1lp}
Let $\f_1,\f_2\in
\cC^\infty_0(\rd)$, with
$\f_1(0)=\f_2(0)=1$,
$\f_1\geq0,\ \f_2\geq0$.
Assume that for some $1\leq
q,r\leq\infty$ and every
$a\in L^q(\rdd)$ the operator
$\gaw$ is bounded on
$L^r(\R^d)$, or that the
estimate
\begin{equation}\label{a00lp}
\|\gaw f\|_{r}\leq C
\|a\|_{q}\|f\|_{r},\,\quad
\forall f\in \cS(\R^d),\
\forall a\in\cS(\R^{2d}),
\end{equation}
holds. Then
\begin{equation}\label{a01lp}
\frac{1}{q}\geq\left|\frac{1}{r}-\frac{1}{2}\right|.
\end{equation}
\end{corollary}

\end{document}